\documentclass [11pt]{article}
\usepackage{amssymb,amsmath,comment,amsthm}

\def\N{\mathbb{N}}

\def\Z{\mathbb{Z}}

\def\F{\mathbb{F}}

\newtheorem{theorem}{Theorem}[section]
\newtheorem{proposition}[theorem]{Proposition}
\newtheorem{corollary}[theorem]{Corollary}
\newtheorem{lemma}[theorem]{Lemma}
\newtheorem{definition}[theorem]{Definition}
\newtheorem{remark}[theorem]{Remark}

\newtheorem{examples}[theorem]{Examples}

\begin{document}
\title{All even (unitary) perfect polynomials over $\F_2$ with only Mersenne primes as odd divisors}
\author{Gallardo, Luis H., Rahavandrainy, Olivier,  \\
Univ Brest, UMR CNRS 6205, \\
Laboratoire de Math\'ematiques de Bretagne Atlantique\\
{\small{e-mail : Luis.Gallardo@univ-brest.fr
- Olivier.Rahavandrainy@univ-brest.fr}}}
\maketitle
\begin{itemize}
\item[a)]
Running head: Mersenne and perfect polynomials
\item[b)]
\item[c)]
\item[d)]
Corresponding author: Luis H. Gallardo
\end{itemize}

\newpage~\\
{\bf{Abstract}}\\
We address an arithmetic problem in the ring $\F_2[x]$ related to the fixed points of the sum of divisors function. We study some binary polynomials $A$ such that $\sigma(A)/A $ is still a binary polynomial. Technically, we prove that
the only (unitary) perfect polynomials over $\F_2$  that are products of $x$, $x+1$ and of Mersenne primes are precisely the nine (resp. nine ``classes'') known ones. This follows from a new result about the factorization of $M^{2h+1} +1$, for a Mersenne prime $M$ and for a positive integer $h$.

{\section{Introduction}}
Let $A \in \F_2[x]$ be a nonzero polynomial. We say that $A$ is \emph{even} if it has a linear factor and that it is
\emph{odd} otherwise. We define a
\emph{Mersenne prime} (polynomial) over $\F_2$ as an irreducible polynomial of the form
$1+x^a(x+1)^b$, for some positive integers $a,b$. The name come as an analogue of the integral Mersenne primes, taking $x^a(x+1)^b$ as an analogue of the prime power $2^{a+b}$.
As over the integers, we say that a divisor $d$ of $A$ is \emph{unitary} if $\displaystyle{\gcd(d,A/d) = 1}$.
Let $\omega(A)$ denote the number of distinct irreducible (or
\emph{prime}) factors of $A$ over $\F_2$ and let $\sigma(A)$ (resp. $\sigma^*(A)$) denote the sum of all (unitary) divisors of $A$ (including $1$ and $A$). Both $\sigma$ and $\sigma^*$ are multiplicative functions. If $\sigma(A) = A$ (resp. $\sigma^*(A)=A$), then we
say that $A$ is (\emph{unitary}) \emph{perfect}. Finally, we say that a (unitary) perfect polynomial is \emph{indecomposable} if it is not a product of two coprime nonconstant (unitary) perfect polynomials.

We can also consider a perfect polynomial ($A \in \F_2[x]$ such that $A$ divides $\sigma(A)$) as an analogue of a multiperfect number (a positive integer which divides the sum of all its divisor). 
It might have some interest to observe that most known multiperfect numbers (see OEIS sequence A007691) appear to be divisible by a Fermat prime or by a Mersenne prime.

The notion of (unitary) perfect polynomials is introduced in \cite{Canaday}
by E. F. Canaday in $1941$ and extended by J. T. B. Beard
Jr. et al.
in several directions (\cite{Beard2}, \cite{BeardU}, \cite{Beard}). Later research in the subject
by Gallardo and Rahavandrainy (\cite{Gall-Rahav5}, \cite{Gall-Rahav8})
allows us to be able to better describe the properties of such polynomials.

The known perfect polynomials are the following:\\
- the so-called ``trivial'' ones, of the form
$(x^2+x)^{2^n-1}$, for some positive integer $n$,\\
- nine others which are already characterized (\cite{Gall-Rahav12}, Theorem 1.1),\\
- and the last two which are divisible by a non-Mersenne prime.\\

Any unitary perfect polynomial is even (Lemma \ref{uperfectiseven}). The known ones, which are only divisible by Mersenne primes (as odd factors), belong to nine equivalence classes (see Lemma \ref{equivalenceclass}). There are several (perhaps, infinitely many) of such classes (see \cite{BeardU} and \cite{Rahav}).\\

The paper consist of two major results stated in Theorem \ref{result3} and Theorem \ref{result0}.
Theorem \ref{result3} significantly improves on the results of \cite[Theorems 1.1 and 1.3]{Gall-Rahav12}, because (in our main result in this paper) there are no more conditions asked on the powers of the $M_j$.
The proof of Theorem \ref{result3} is obtained from new results given in Theorem \ref{result0}, which in turn, extends  recent results in \cite[Theorem 1.4]{Gall-Rahav-mersenn}.\\

It is convenient to fix some notations.\\
\\
{\bf{Notations}}\\
$\bullet$ The set of integers (resp. of nonnegative
integers, of positive integers) is denoted by $\Z$ (resp. $\N$, $\N\sp{*}$).\\
$\bullet$ For $S, T \in \F_2[x]$ and for $m \in \N^*$, $S^m \mid T$ (resp. $S^m \| T$) means that $S$ divides $T$ (resp. $S^m \mid T$ but $S^{m+1} \nmid T$). We also denote by $\overline{S}$ the polynomial defined as $\overline{S}(x) = S(x+1)$ and by
$val_x(S)$ $($resp. $val_{x+1}(S))$ the valuation of $S$, at $x$ $($resp. at $x+1)$.\\
$\bullet$ We put
$$\begin{array}{l}
M_j= 1+x(x+1)^j, j \in \{1,2,3\},\\
T_1 =x^2(x+1)M_1, T_2=\overline{T_1}, \\
T_3 = x^4(x+1)^3M_3, T_4 =\overline{T_3},\
T_5 = x^4(x+1)^4M_3\overline{M_3} = \overline{T_5},\\
T_6 = x^6(x+1)^3M_2\overline{M_2}, T_7= \overline{T_6},\\
T_8 = x^4(x+1)^6M_2\overline{M_2} M_3 \text{ and } T_9 = \overline{T_8},\\
B_1 = x^3(x+1)^3M_1^2, \ B_2 = x^3(x+1)^2M_1, \ B_3 = x^5(x+1)^4M_3, \\
B_4 = x^7(x+1)^4M_2 \overline{M_2}, \ B_5 = x^5(x+1)^6M_1^2M_3, \ B_6 = x^5(x+1)^5M_3 \overline{M_3}, \\
B_7 = x^7(x+1)^7M_2^2 {\overline{M_2}}^2, \ B_8 = x^7(x+1)^6M_1^2 M_2 \overline{M_2}, \ B_9 = x^7(x+1)^5M_2 \overline{M_2} \ \overline{M_3},\\
\text{${\mathcal{M}}= \{M_1, M_2,\overline{M_2}, M_3, \overline{M_3} \}$, ${\mathcal{P}}= \{T_1,\ldots, T_9\}$ and ${\mathcal{P}}_u= \{B_1,\ldots, B_9\}$.}
\end{array}$$
$\bullet$ Finally, we denote by $\Delta$, the set of primes $p$ such that $p$ is a Mersenne prime or the order of $2$ in $\F_p \setminus\{0\}$, $ord_p(2)$, is divisible by~$8$. In particular,  $\Delta$ contains all Fermat primes greater than $5$.\\

Throughout this paper, we {\bf{always}} suppose that any (unitary) perfect polynomial is {\bf{indecomposable}}. We have often used Maple software for computations.
Our main results are the following.\\
\begin{theorem} \label{result3}
Let $\displaystyle{A = x^a(x+1)^b \prod_{i\in I} P_i^{h_i} \in \F_2[x]}$ be such that each $P_i$ is a Mersenne prime and $a,b,h_i \in \N^{*}$. Then
$A$ is perfect $($resp.  unitary perfect$)$ if and only if $A \in {\mathcal{P}}$ $($resp. $A = B^{2^n}$ for some $n \in \N$ and $B \in {\mathcal{P}}_u)$.
\end{theorem}

\begin{theorem} \label{result0}
Let $h \in \N^*$ and let $M \in \F_2[x]$ be a Mersenne prime. Then in the following cases, $\sigma(M^{2h})$ is divisible by a non-Mersenne prime:\\
(i) $(M \in \{M_1, M_3, \overline{M_3} \})$ or $(M \in \{M_2, \overline{M_2}\}$ and $h \geq 2)$.\\
(ii) $M \not\in {\mathcal{M}}$ and $2h+1$ is divisible by a prime number $p \in \Delta \setminus \{ 7 \}$.
\end{theorem}

\section{Proof of Theorem \ref{result3}} \label{proof3}
Sufficiencies are obtained by direct computations. For the necessities, we shall apply Lemmas \ref{oldresult1} and \ref{oldresult2}, Propositions \ref{caseperfect} and \ref{caseunitperfect}. We fix:
$$\text{$\displaystyle{A= x^a(x+1)^b \prod_{i \in I}  P_i^{h_i}= A_1 A_2}$, where $a,b, h_i \in \N$, $P_i$ is a Mersenne prime,}$$ $\text{$\displaystyle{A_1 = x^a(x+1)^b \prod_{P_i \in {\mathcal{M}}} P_i^{h_i}}$  and $\displaystyle{A_2 = \prod_{P_j \not\in {\mathcal{M}}} P_j^{h_j}}$.}$

\begin{lemma} \label{sigmaPjhj}
If $A$ is perfect $($resp. unitary perfect$)$, then $\sigma(x^a)$, $\sigma((x+1)^b)$ and $\sigma({P_i}^{h_i})$, for any $i \in I$ $($resp. $\sigma^*(x^a)$, $\sigma^*((x+1)^b)$, $\sigma^*({P_i}^{h_i}))$ are only divisible by $x$, $x+1$ or by Mersenne primes.
\end{lemma}
\begin{proof}
Since $\sigma$ and $\sigma^*$ are multiplicative, $\displaystyle{\sigma(A) = \sigma(x^a) \sigma((x+1)^b)  \prod_{i \in I} \sigma({P_i}^{h_i})}$ (resp. $\displaystyle{\sigma^*(A) = \sigma^*(x^a) \sigma^*((x+1)^b) \prod_{i \in I} \sigma^*({P_i}^{h_i})}$). Any divisor of $\sigma(x^a)$, $\sigma((x+1)^b)$ and $\sigma({P_i}^{h_i})$ (resp. of $\sigma^*(x^a)$, $\sigma^*((x+1)^b)$, $\sigma^*({P_i}^{h_i})$) divides $\sigma(A) = A$ (resp.  $\sigma^*(A) = A$).
\end{proof}
\begin{lemma} [Lemma 2 in \cite{Beard}]  \label{upcriteres}
A polynomial $S$ is $($unitary$)$ perfect if and only if for any irreducible polynomial $P$ and for any $m_1, m_2 \in \N^*$, we have
$(P^{m_1} \| S, P^{m_2} \| \sigma(S)) \Rightarrow m_1 = m_2$ $($resp. $(P^{m_1} \| S, P^{m_2} \| \sigma^*(S)) \Rightarrow m_1 = m_2)$.
\end{lemma}
\begin{examples}  [useful for Propositions \ref{caseperfect} and \ref{caseunitperfect}] \label{upexamples}~\\
(i) The polynomial $S_1 = x^{13} (x+1)^2 M_1^3 {M_2}^2 {\overline{M_2}}^{\ 2} M_3 \overline{M_3}$ is not perfect because $x^{13} \| S_1$ and $x^7 \| \sigma(S_1)$.\\
(ii) The polynomial $S_2 = x^{14} (x+1)^7 {M_1}^2 {M_2}^3 {\overline{M_2}}^{\ 3} M_3 \overline{M_3}$ is not unitary perfect since $x^{14} \| S_2$ and $x^{10} \| \sigma^*(S_2)$.
\end{examples}
\subsection{Case of perfect polynomials} \label{case-perfect}

\begin{lemma} [Theorem 1.1 in \cite{Gall-Rahav12}] \label{oldresult1}
If $h_i = 2^{n_i}-1$ for any $i \in I$, then $A \in {\mathcal{P}}$.
\end{lemma}
We get from Theorem 8 in \cite{Canaday} and from Theorem \ref{result0}.
\begin{lemma} \label{canadayperf}
(i) If $h \in \N^*$ and if $\sigma(x^{2h})$ is only divisible by Mersenne primes, then $2h \in \{2,4,6\}$ and all its divisors lie in ${\mathcal{M}}$. More precisely, $\sigma(x^2) = M_1 = \sigma((x+1)^2),\ \sigma(x^4) = M_3,\  \sigma((x+1)^4) = \overline{M_3}$ and  $\sigma(x^6) = M_2 \overline{M_2} =\sigma((x+1)^6)$.\\
(ii) Let $M \in {\mathcal{M}}$ and $h \in \N^*$ be such that $\sigma(M^{2h})$ is only divisible by Mersenne primes, then $2h=2$, $M \in \{M_2, \overline{M_2}\}$ and $\sigma(M^2) \in \{M_1M_3, M_1\overline{M_3}\}$.
\end{lemma}
We dress from Lemma \ref{canadayperf}, the following tables of all the forms of $a$, $b$, $P_i$ and $h_i$ which satisfy Lemma \ref{sigmaPjhj}, if $P_i \in {\mathcal{M}}$ and if $h_i \not= 2^{n_i}-1$.\\
\\
$\begin{array}{|l|l|}
\hline
&\\
a&\sigma(x^a)\\
\hline
&\\
3\cdot 2^n-1& (x+1)^{2^n-1} {M_1}^{2^n}\\
\hline
&\\
5\cdot 2^n-1& (x+1)^{2^n-1} {M_3}^{2^n}\\
\hline
&\\
7\cdot 2^n-1& (x+1)^{2^n-1} {M_2}^{2^n} \ {\overline{M_2}}^{\ 2^n}\\
\hline
\end{array}$ \ \  $\begin{array}{|l|l|}
\hline
&\\
b&\sigma((x+1)^b)\\
\hline
&\\
3\cdot 2^m-1& x^{2^m-1} {M_1}^{2^m}\\
\hline
&\\
5\cdot 2^m-1& x^{2^m-1} {\overline{M_3}}^{\ 2^m}\\
\hline
&\\
7\cdot 2^m-1& x^{2^m-1} {M_2}^{2^m} \ {\overline{M_2}}^{\ 2^m}\\
\hline
\end{array}$\\
\\
$\begin{array}{|l|l|l|}
\hline
&&\\
P_i&h_i&\sigma({P_i}^{h_i})\\
\hline
&&\\
M_2&3\cdot 2^{n_i}-1& (1+M_2)^{2^{n_i}-1} {M_1}^{2^{n_i}} \ {\overline{M_3}}^{\ 2^{n_i}}\\
\hline
&&\\
\overline{M_2}&3\cdot 2^{n_i}-1& (1+\overline{M_2})^{2^{n_i}-1} {M_1}^{2^{n_i}} \ {M_3}^{2^{n_i}}\\
\hline
\end{array}$
\begin{corollary} \label{ifM2divsor}
Suppose that $A_1$ is perfect. Then, neither $M_2$ nor $\overline{M_2}$ divides $\sigma({P_i}^{h_i})$ if $P_i \in {\mathcal{M}}$. Moreover, $\overline{M_2}$ divides $A_1$ whenever $M_2$ divides $A_1$ and their exponents $($in $A_1)$ are equal.
\end{corollary}
\begin{proof}
The first statement follows from Lemma \ref{canadayperf}-(ii). Now, if $M_2$ divides $A_1 = \sigma(A_1)$, then $M_2$ divides $\displaystyle{\sigma(x^a) \ \sigma((x+1)^b) \prod_{P_i \in {\mathcal{M}}} \sigma(P_i^{h_i})}$. 
Hence, $M_2$ divides $\sigma(x^a) \sigma((x+1)^b)$. The above tables show that $a$ or $b$ is of the form $7 \cdot 2^n-1$, where $n \in \N$. So, $\overline{M_2}$ divides $\sigma(A_1) = A_1$. It suffices to consider two cases. If $a=7 \cdot 2^{n}-1$ and $b=7 \cdot 2^m-1$, then ${M_2}^{\ell} \| A_1$ and ${\overline{M_2}}^{\ \ell} \| A_1$, with $\ell = 2^n + 2^m$. If $a=7 \cdot 2^{n}-1$ and $(b=3 \cdot 2^m-1$ or $b=5 \cdot 2^m-1)$, then ${M_2}^{\ell} \| A_1$ and ${\overline{M_2}}^{\ \ell} \| A_1$, with $\ell = 2^n$.
\end{proof}

\begin{lemma} \label{divisorsigmA1}~\\
If $P$ is a Mersenne prime divisor of $\sigma(A_1)$, then $P, \overline{P} \in \{M_1,M_2, M_3\}$.
\end{lemma}
\begin{proof}
One has: $\displaystyle{\sigma(A_1) = \sigma(x^a) \sigma((x+1)^b) \prod_{P_i \in {\mathcal{M}}} \sigma(P_i^{h_i})}$. If $P$ divides $\sigma(x^a) \sigma((x+1)^b)$, then $P \in {\mathcal{M}}$ by Lemma \ref{canadayperf}-(i). If $P$ divides $\sigma(P_i^{h_i})$ with $P_i \in {\mathcal{M}}$, then 
$P_i \in \{M_2, \overline{M_2} \}$, $(h_i = 2$ or $h_i$ is of the form $3 \cdot 2^{n_i}-1)$ and $P, \overline{P} \in \{M_1, M_3\}$ (see the above tables).
\end{proof}

\begin{lemma} \label{gcdMjsigmA1}
If $A$ is perfect, then $A=A_1$.
\end{lemma}
\begin{proof}
We claim that $A_2 = 1$. Let $P_j \not\in {\mathcal{M}}$ and $Q_i \in {\mathcal{M}}$. Then, $P_j$ divides neither $\sigma(x^a)$, $\sigma((x+1)^b)$ nor $\sigma(Q_i^{h_i})$. Thus $\gcd(P_j^{h_j}, \sigma(A_1)) = 1$. \\
Observe that $P_j^{h_j}$ divides $\sigma(A_2)$ because $P_j^{h_j}$ divides $A = \sigma(A)=\sigma(A_1) \sigma(A_2)$. Hence, $A_2$ divides $\sigma(A_2)$. So, $A_2$ is perfect and it is equal to $1$, $A$ being indecomposable.
\end{proof}

\begin{proposition} \label{caseperfect}
If $A_1$ is perfect, then $h_j=2^{n_j}-1$ for any $P_j \in {\mathcal{M}}$.
\end{proposition}
\begin{proof}  We refer to Tables at the beginning of this section.\\
(i) Suppose that $P_j \not\in  \{M_2, \overline{M_2}\}$. If $h_j$ is even, then $\sigma(P_j^{h_j})$ is divisible by a non-Mersenne prime.
It contradicts Lemma \ref{sigmaPjhj}. If $hj=2^{n_j}u_j-1$ with $u_j \geq 3$ odd, then
$\sigma(P_j^{h_j}) = (1+P_j)^{2^{n_j}-1}\cdot (1+P_j+\cdots + P_j^{u_j-1})^{2^{n_j}}$. Since $1+P_j+\cdots + P_j^{u_j-1} = \sigma({P_j}^{u_j-1})$ is divisible by a non-Mersenne prime, we also get a contradiction to Lemma \ref{sigmaPjhj}.\\
(ii) If $P_j \in \{M_2,\overline{M_2} \}$ and ($h_j$ is even or it is of the form $2^{n_j}u_j-1$, with $u_j \geq 3$ odd and $n_j \geq 1$),
then Corollary \ref{ifM2divsor} implies that there exists $\ell \in \N^*$ such that ${M_2}^{\ell} \| A_1$ and  ${\overline{M_2}}^{\ \ell} \| A_1$. Recall that $\sigma({M_2}^2) = M_1 \overline{M_3}$ and $\sigma({\overline{M_2}}^{\ 2}) = M_1 M_3$. We proceed as in the proof of Corollary \ref{ifM2divsor}. It suffices to distinguish four cases which give contradictions.\\
$\bullet$ Case 1: $a=7 \cdot 2^{n}-1$ and $b=7 \cdot 2^m-1$\\
One has $\ell = 2^n + 2^m$ and neither $M_1$ nor $M_3$ divides $\sigma(x^a) \ \sigma((x+1)^b)$.\\
If $h_j$ is even, then $h_j = 2 = \ell$. So, $n=m=0$, ${M_1}^2 \| \sigma(A_1) = A_1$. It contradicts the part (i) of our proof.\\
If $h_j = 2^{n_j}u_j-1$ with $u_j \geq 3$ odd and $n_j \geq 1$, then $u_j = 3$ and ${M_1}^{2 \cdot 2^{n_j}} \| A_1$.\\
$\bullet$ Case 2: $a=7 \cdot 2^{n}-1$ and $b=5 \cdot 2^m-1$\\
One has $\ell = 2^n$ and $M_1 \nmid \sigma(x^a) \sigma((x+1)^b)$. If $h_j$ is even, then $2^n = \ell = h_j =  2$. So, $n=1$ and ${M_1}^2 \| A_1$.
If $h_j = 2^{n_j}u_j-1$, with $u_j \geq 3$ odd and $n_j \geq 1$, then $u_j = 3$ and $2^n = \ell = h_j =  3 \cdot 2^{n_j}-1$. It is impossible.\\
$\bullet$ Case 3: $a=7 \cdot 2^{n}-1$, $b=3 \cdot 2^m-1$ and $h_j$ is even\\
As above, $2^n=\ell = h_j=2$, ${M_1}^{2^m}$ divides $\sigma((x+1)^b)$ and ${M_1}^{2^n+2^m}$ divides $\sigma(A_1)=A_1$. So, $n=1$ and ${M_1}^{2^m+2} \| A_1$. Thus, the part (i) implies that $m=0$. Hence,
$A_1 = S_1=x^{13} (x+1)^2 M_1^3 {M_2}^2 {\overline{M_2}}^{\ 2} M_3 \overline{M_3}$ which is not perfect (see Examples \ref{upexamples}).\\
$\bullet$ Case 4: $a=7 \cdot 2^{n}-1$, $b=3 \cdot 2^m-1$, $h_j = 2^{n_j}u_j-1$, $u_j \geq 3$ odd, $n_j \geq 1$\\
One has $u_j = 3$ and $2^n = \ell = h_j =  3 \cdot 2^{n_j}-1$. It is impossible.
\end{proof}
Lemma \ref{gcdMjsigmA1}, Proposition \ref{caseperfect} and Lemma \ref{oldresult1} imply
\begin{corollary} \label{A1perfect}
If $A$ is perfect, then $A=A_1 \in {\mathcal{P}}$.
\end{corollary}

\subsection{Case of unitary perfect ($u. p$) polynomials} \label{case-unit-perfect}
Similar arguments give Proposition \ref{caseunitperfect} which finishes our proof.
\begin{lemma} \label{unitary1}
Let $S \in \F_2[x]$ be an irreducible polynomial. Then, for any $n, u \in \N$ with $u$ odd,
$\sigma^*(S^{2^n u}) = (1+S)^{2^n} (\sigma(S^{u-1}))^{2^n}$.
\end{lemma}
We may dress the following tables, from Lemmas \ref{sigmaPjhj}, \ref{canadayperf} and \ref{unitary1}.\\
\\
$\begin{array}{|l|l|}
\hline
&\\
a&\sigma^*(x^a)\\
\hline
&\\
3\cdot 2^n& (x+1)^{2^n} {M_1}^{2^n}\\
\hline
&\\
5\cdot 2^n& (x+1)^{2^n} {M_3}^{2^n}\\
\hline
&\\
7\cdot 2^n& (x+1)^{2^n} {M_2}^{2^n} \ {\overline{M_2}}^{\ 2^n}\\
\hline
\end{array}$ \ \  $\begin{array}{|l|l|}
\hline
&\\
b&\sigma^*((x+1)^b)\\
\hline
&\\
3\cdot 2^m& x^{2^m} {M_1}^{2^m}\\
\hline
&\\
5\cdot 2^m& x^{2^m} {\overline{M_3}}^{\ 2^m}\\
\hline
&\\
7\cdot 2^m& x^{2^m} {M_2}^{2^m} \ {\overline{M_2}}^{\ 2^m}\\
\hline
\end{array}$\\
\\
$\begin{array}{|l|l|l|}
\hline
&&\\
P_i&h_i&\sigma^*({P_i}^{h_i})\\
\hline
&&\\
M_2&3\cdot 2^{n_i}&(1+M_2)^{2^{n_i}} {M_1}^{2^{n_i}} \  {\overline{M_3}}^{\ 2^{n_i}}\\
\hline
&&\\
\overline{M_2}&3\cdot 2^{n_i}& (1+\overline{M_2})^{2^{n_i}} {M_1}^{2^{n_i}} \ {M_3}^{2^{n_i}}\\
\hline
\end{array}$

\begin{lemma} \label{uperfectiseven}
Let $C \in \F_2[x] \setminus\{1\}$ be $u. p$. Then $C$ is even, $\overline{C}$ and $C^{2^r}$ are also $u. p$, for any $r \in \N$.
\end{lemma}
\begin{proof}
If $D$ is a divisor of $C$, then $\overline{D}$ divides  $\overline{C}$ and $D^{2^r}$ divides $C^{2^r}$. Thus,
$\sigma^*(\overline{C}) = \overline{\sigma^*(C)} = \overline{C}$ and $\sigma^*(C^{2^r}) = (\sigma^*(C))^{2^r} = C^{2^r}$.\\
It remains to prove that $C$ is even. Consider an irreducible divisor $P$ of $C$ and $k \in \N^*$ such that $P^k \| C$. The polynomial $1+P$ is even and divides $1+P^k = \sigma^*(P^k)$. So, $1+P$ divides $\sigma^*(C) = C$.
\end{proof}
\begin{definition}
We denote by $\sim$ the relation on $\F_2[x]$ defined as:  $S \sim T$ if there exists $\ell \in \Z$ such that
$S = T^{2^{\ell}}.$
\end{definition}
\begin{lemma} {\rm{(\cite{BeardU2}, Section 2})} \label{equivalenceclass}~\\
The relation $\sim$ is an equivalence relation on $\F_2[x]$. Each equivalence class contains a unique polynomial $B$ which is not a square, with $val_x(B) \leq val_{x+1}(B)$.
\end{lemma}

\begin{lemma} [Theorem 1.3 in \cite{Gall-Rahav12}] \label{oldresult2}
 If $h_i = 2^{n_i}$ for any $i \in I$, then $A$ $($or $\overline{A})$ is of the form $B^{2^n}$, where $B \in {\mathcal{P}}_u$.
\end{lemma}

\begin{proposition} \label{caseunitperfect}
(i) If $A$ is $u. p$, then $A = A_1$.\\
(ii) If $A_1$ is $u. p$, then $h_j=2^{n_j}$ for any $P_j \in {\mathcal{M}}$.\\
(iii) If $A$ is $u. p$, then $A$ or $\overline{A}$ is of the form $B^{2^n}$, where $B \in {\mathcal{P}}_u$.
\end{proposition}
\begin{proof} The proof of (i) is analogous to that of Lemma \ref{gcdMjsigmA1}. The statement (iii) follows from (i), (ii) and Lemma \ref{oldresult2}.
We only sketch the proof of (ii). \\
Set $h_j = 2^{n_j} u_j$, where $u_j$ is odd and $n_j \geq 0$. \\
- Suppose that $P_j \not\in  \{M_2, \overline{M_2}\}$. If $u_j \geq 3$, then $\sigma(P_j^{u_j -1})$ and thus $\sigma^*(P_j^{h_j})$ are divisible by a non-Mersenne prime. It contradicts Lemma \ref{sigmaPjhj}.\\
- If $P_j \in  \{M_2, \overline{M_2}\}$ and if $u_j \geq 3$, then $u_j = 3$ and $(a$ or $b$ is of the form $7 \cdot 2^n$). Recall that $\sigma^*({M_2}^3) = (1+M_2) M_1 \overline{M_3}$ and $\sigma^*({\overline{M_2}}^{\ 3}) = (1+ \overline{M_2}) M_1 M_3$. 
We consider two cases. The first gives non unitary perfect polynomials whereas the second leads to a contradiction.\\
$\bullet$ Case 1: $a=7 \cdot 2^{n}$ and $b=7 \cdot 2^m$, with $n, m\geq 0$\\
One has ${M_2}^{\ell} \| A_1$ and ${\overline{M_2}}^{\ \ell} \| A_1$, with $\ell = 2^n + 2^m$. Neither $M_1$ nor $M_3$ divides $\sigma(x^a) \ \sigma((x+1)^b)$.\\
Thus, $3 \cdot 2^{n_j} = h_j= \ell = 2^n + 2^m$. So, $(n=m+1$ and $n_j=m)$ or $(m=n+1$ and $n_j=n)$.
Therefore, $({M_1}^{2})^{2^{n_j}}$, ${M_3}^{2^{n_j}}$ and ${\overline{M_3}}^{\ 2^{n_j}}$ divide $\sigma^*(M_2^{h_j}) \sigma^*({\overline{M_2}}^{\ h_j})$ and they divide $\sigma^*(A_1) = A_1$. Thus,
$A_1 ={S_2}^{2^m}$  or $A_1 ={\overline{S_2}}^{\ 2^n}$ where $S_2 = x^{14} (x+1)^7 {M_1}^2 {M_2}^3 {\overline{M_2}}^{\ 3} M_3 \overline{M_3}$. 
In both cases, $A_1$ is not unitary perfect because $S_2$ is not u.p (Examples \ref{upexamples}).\\
$\bullet$ Case 2: $a=7 \cdot 2^{n}$ and $(b=5 \cdot 2^m$ or $b=3 \cdot 2^m$), with $n, m\geq 0$\\
One has $\ell = 2^n$. So, we get the contradiction: $3\cdot 2^{n_j} = h_j= \ell = 2^n$.
\end{proof}

\section{Proof of Theorem \ref{result0}} \label{proof0}
We mainly prove Theorem \ref{result0} by contradiction (to Corollary \ref{sommaieven}). Lemma \ref{omegasigmM2h} states that $\sigma(M^{2h})$ is square-free, for any $h \in \N^*$.\\
\\
We set
$M=x^a(x+1)^b+1$, $U_{2h}=\sigma(\sigma(M^{2h}))$ and 
\begin{equation}\label{assume}
\text{$\sigma(M^{2h})= \displaystyle{\prod_{j\in J} {P_j}}$, $P_j = 1+x^{a_j}(x+1)^{b_j}$ irreducible, $P_i \not= P_j$ if $i \not= j.$}
\end{equation}
By Lemma \ref{p-reduction}, if there exists a prime divisor $p$ of $2h+1$ such that $\sigma(M^{p-1})$  is divisible by a non-Mersenne prime, then  $\sigma(M^{2h})$ is also divisible by a non-Mersenne. Therefore, it suffices to consider that $2h+1 = p$ is a prime number, except for $p = 3$ with $M \in \{M_2, \overline{M_2}\}$
(see Section \ref{caseM2}).

\subsection{Useful facts}\label{preliminaire}
For $S \in \F_2[x] \setminus \{0,1\}$, of degree $s$, we denote by $\alpha_l(S)$ the coefficient of $x^{s-l}$ in $S$, $0\leq l \leq s.$ One has: $\alpha_0(S) =1$.

\begin{lemma} [Lemmas 4.6 and 4.8 in \cite{Gall-Rahav-mersenn}] \label{omegasigmM2h}~\\
The polynomial $\sigma(M^{2h})$ is square-free and $M \not=M_1$.
\end{lemma}

\begin{lemma} [Theorem 1.4 in \cite{Gall-Rahav-mersenn}] \label{oldresult3}
Let $h\in \N^*$ be such that $p=2h+1$ is prime and let $M$ be a Mersenne prime such that $M \not\in \{M_2, \overline{M_2}\}$ and $\omega(\sigma(M^{2h})) = 2$. Then, $\sigma(M^{2h})$ is divisible by a non-Mersenne prime.
\end{lemma}
The lemma below generalizes Lemma 4.10 in \cite{Gall-Rahav-mersenn} (with an analogous proof).

\begin{lemma} \label{p-reduction}
If $k$ is a divisor $($prime or not$)$ of $2h+1$, then $\sigma(M^{k-1})$ divides $\sigma(M^{2h})$.
\end{lemma}
We sometimes apply Lemmas \ref{lesalfal} and \ref{alfasigmM2h} without explicit mentions.

\begin{lemma} \label{lesalfal}
Let $S \in \F_2[x]$ be such that $s = \deg(S) \geq 1$ and $l,t,r,r_1,\ldots, r_k \in \N$ be such that $r_1>\cdots >r_k$, $t \leq k,
r_1-r_t \leq l \leq r \leq s.$  Then\\
(i) $\alpha_l[(x^{r_1} + \cdots +x^{r_k})S] = \alpha_l(S) + \alpha_{l-(r_1-r_2)}(S)+\cdots + \alpha_{l-(r_1-r_t)}(S)$.\\
(ii) $\alpha_l(\sigma(S))=\alpha_l(S)$ if any divisor of $S$ has degree at least $r+1$.
\end{lemma}

\begin{proof}
The equality in (i) (resp. in (ii)) follows from the definition of $\alpha_l$ (resp. from the fact: $\sigma(S) = S + T$, where $\deg(T) \leq \deg(S)-r-1$).
\end{proof}

\begin{corollary} \label{sommaieven}
(i) The integers $\displaystyle{u=\sum_{j \in J} a_j}$ and $\displaystyle{v=\sum_{j \in J} b_j}$ are both even.\\
(ii) The polynomial $U_{2h}$ splits $($over $\F_2)$ and it is a square.\\
(iii) The polynomial $\sigma(M^{2h})$ is reducible.
\end{corollary}

\begin{proof}
(i) See \cite[Corollary 4.9]{Gall-Rahav-mersenn}. For (ii), Assumption (\ref{assume}) implies that $\displaystyle{U_{2h}= \sigma(\sigma(M^{2h}))=\sigma(\prod_{j\in J} {P_j}) =  \prod_{j \in J} x^{a_j}(x+1)^{b_j} = x^u(x+1)^v},$ where $u$ and $v$ are both even.\\
(iii) If $\sigma(M^{2h}) = Q$ is irreducible, then $U_{2h} = 1+Q$ is not a square.
\end{proof}

\begin{lemma} \label{alfasigmM2h}
One has $\alpha_l(\sigma(M^{2h}))=\alpha_l(M^{2h})$ if $l \leq a+b-1$ and
$\alpha_l(\sigma(M^{2h})) = \alpha_l(M^{2h}+M^{2h-1})$ if $a+b \leq l \leq 2(a+b)-1$.
\end{lemma}
\begin{proof}
Since $\sigma(M^{2h}) = M^{2h} + M^{2h-1} + T$, with $\deg(T) \leq (a+b)(2h-2)=2h(a+b)-2(a+b)$, Lemma \ref{lesalfal}-(ii) implies that
$\alpha_l(\sigma(M^{2h})) = \alpha_l(M^{2h})$ if $l \leq a+b-1$ and
$\alpha_l(\sigma(M^{2h})) = \alpha_l(M^{2h}+M^{2h-1})$
if $a+b\leq l \leq 2(a+b)-1$.
\end{proof}

\begin{lemma} \label{phiandirreduc}
Denote by $N_2(m)$ the number of irreducible polynomials over $\F_2$, of degree $m \geq 1$. Then\\
(i) $\displaystyle{N_2(m) \geq \frac{2^m - 2(2^{m/2} - 1)}{m}}$,\\
(ii) $\varphi(m) < N_2(m)$ if $m\geq 4$, where $\varphi$ is the Euler totient function,\\
(iii) For each $m \geq 4$, there exists an irreducible polynomial of degree $m$, which is not a Mersenne prime.
\end{lemma}

\begin{proof}
(i) See Exercise 3.27, p. 142 in \cite{Rudolf}.\\
(ii) If $m \in \{4, 5\}$, then direct computations give  
$\varphi(4) = 2,\ N_2(4) = 3$ and $\varphi(5) = 4, \ N_2(5) = 6$.\\
Now, suppose that $m \geq 6$. Consider the function $f(x)=2^x - 2(2^{x/2} - 1) -x^2$, for $x\geq 6$. The derivative of $f$ is a positive function. So,  $f(x) \geq f(6) > 0$ and $x < \displaystyle{\frac{2^x - 2(2^{x/2} - 1)}{x}}$.
Thus,
$\varphi(m) \leq m < \displaystyle{\frac{2^m - 2(2^{m/2} - 1)}{m}} \leq N_2(m)$.\\
(iii) We remark that if $1+x^c(x+1)^d$ is a Mersenne prime, then $\gcd(c,d)=1$. So, $\gcd(c,c+d)=1$.
Therefore,
the set ${\mathcal{M}}_m$ of Mersenne primes of degree $m$ is a subset of $\{ x^c(x+1)^{m-c} +1: 1 \leq c \leq m, \ \gcd(c,m)=1\}.$
Thus, $$\# {\mathcal{M}}_m \leq \# \{c: 1\leq c \leq m,\ \gcd(c,m)=1\} = \varphi(m).$$
Hence, 
there exist at least $N_2(m) - \varphi(m)$ irreducible non-Mersenne polynomials, with $N_2(m) - \varphi(m) \geq 1$, by (ii).
\end{proof}

\begin{lemma} \label{ord(2)divise}
For any $j \in J$, $ord_p(2)$ divides $a_j+b_j = \deg(P_j)$.
\end{lemma}

\begin{proof}
Set $\displaystyle{d=\gcd_{i\in J} (a_i+b_i)}$. By Lemma 4.13 in \cite{Gall-Rahav-mersenn}, $p$ divides $2^d-1$. Thus, $ord_p(2)$ divides $d$.
\end{proof}

\begin{lemma} \label{primitive} {\rm{(\cite{Rudolf}, Chap. 2 and 3)}}~\\
Let $q=2^r-1$ be a Mersenne prime number. Then, any irreducible polynomial $P$ of degree $r$ is primitive.  In particular,
each root $\beta$ of $P$ is a primitive element of the field $\F_{2^r}$, so that $\beta$ is of order $q$ in  $\F_{2^r} \setminus \{0\}$.
\end{lemma}
\begin{lemma} \label{reduction2}
Let $P_i = 1+x^{a_i}(x+1)^{b_i}$ be a prime divisor of $\sigma(M^{p-1}),$ where $2^{a_i+b_i}-1 = p_i$ is a prime number.  Then, $p_i = p$ and $\sigma(M^{p-1})$
is divisible by any irreducible polynomial of degree $a_i+b_i$. Furthermore, at least one of those divisors is not a Mersenne prime if $a_i+b_i \geq 4$.
\end{lemma}
\begin{proof}
The polynomial $P_i$ is primitive. If $\alpha$ is a root  of $P_i$,  then $(M^p +1)(\alpha) = 0$ and $M(\alpha) = \alpha^r$ for some $1\leq r \leq p_i-1$. Thus, $1= M(\alpha)^p = \alpha^{rp}$, with $ord(\alpha) = p_i$. So, $p_i$ divides $rp$ and $p_i=p$.\\
Any irreducible polynomial $S$ of degree $a_i+b_i$ is primitive. Let $\beta$ be a root of $S$. One has $ord(\beta) = p_i=p$, $S(\beta)=0$ and $M(\beta) = \beta^s$, for some $1\leq s \leq p_i-1$. Thus, $M(\beta)^{p} = \beta^{ps}= 1$ and $S$ divides $M^p + 1 = x^a(x+1)^b \sigma(M^{p-1})$.\\
The third statement follows from Lemma \ref{phiandirreduc}-(iii).
\end{proof}
\begin{corollary} \label{reduction3}
For any $i \in J$, $a_i+b_i \leq 3$ or $2^{a_i+b_i} -1$ is not prime.
\end{corollary}

\begin{lemma} \label{PandQ}
Let $P, Q \in \F_2[x]$ be such that $\deg(P)=r$, $2^r-1$ is prime, $P \nmid Q(Q+1)$ but $P \mid Q^p +1$. Then $2^r-1=p$.
\end{lemma}
\begin{proof}
The polynomial $P$ is primitive. If $\beta$ is a root of $P$, then $ord(\beta) = 2^r-1$. Moreover, $Q(\beta) \not\in \{0,1\}$ because $P \nmid Q(Q+1)$. Thus, $Q(\beta) = \beta^t$ for some $1\leq t \leq 2^r-2$. Hence, $1=Q(\beta)^p = \beta^{tp}.$ So, $2^r-1$ divides $tp$ and $2^r-1=p$.
\end{proof}
\begin{corollary} \label{notdivisor}
Let $r \in \N^*$ be such that $2^r-1$ is a prime distinct from $p$. Then, no irreducible polynomial of degree $r$ divides $\sigma(M^{p-1})$.
\end{corollary}
\begin{proof}
If $P$ is a prime divisor of $\sigma(M^{p-1})$ with $\deg(P)=r$, then $P$ divides $M^p+1$ and by taking $Q = M$ in the above lemma, we get a contradiction.
\end{proof}

In the following  lemma and  two corollaries, we suppose that $p$ is a Mersenne prime of the form $2^m-1$ (with $m$ prime).

\begin{lemma} \label{polyPandQ}
Let $P, Q \in \F_2[x]$ be such that $P$ is irreducible of degree $m$ and $P \nmid Q(Q+1)$.
Then, $P$ divides $Q^p +1$.
\end{lemma}
\begin{proof}
The polynomial $P$ is primitive. If $\beta$ is a root of $P$, then $ord(\beta) = 2^m-1=p$, $Q(\beta) \not\in \{0,1\}$ because $P \nmid Q(Q+1)$. Thus, $Q(\beta) = \beta^t$ for some $1\leq t \leq p-1$. Hence, $Q(\beta)^p = \beta^{tp} = 1.$ So, $P$ divides $Q^p + 1$.
\end{proof}
\begin{corollary} \label{anydivides}
Any irreducible polynomial $P \not = M$ $($Mersenne or not$)$, of degree $m$, divides $\sigma(M^{p-1})$.
\end{corollary}
\begin{proof}
We may apply Lemma \ref{polyPandQ} with $Q = M$ because $P$ does not divide $x^a(x+1)^bM=M(M+1)=Q(Q+1)$. So, $P$ is odd and it divides $M^p +1 = (M+1) \ \sigma(M^{p-1}) = x^a(x+1)^b \ \sigma(M^{p-1})$.
\end{proof}
\begin{corollary} \label{reduction4}
The polynomial
$M_1$ $($resp. $M_2$, $\overline{M_2})$ divides $\sigma(M^{p-1})$ if and only if $(M \not= M_1$ and $p=3)$
$($resp. $M \not= M_2$ and $p=7$, $M \not= \overline{M_2}$ and $p=7)$.
\end{corollary}
\begin{proof} Apply Corollary \ref{anydivides} with $m\in \{2,3\}$.
\end{proof}
In order to carry on the proof (of Theorem \ref{result0}), we distinguish three cases.
\subsection{Case I:  $M \in \{M_1, M_3, \overline{M_3}\}$}
Lemma \ref{omegasigmM2h} implies that $M \not= M_1$. It suffices to suppose that $M =M_3$.
We refer to Section 5.2 in \cite{Gall-Rahav13}. Put $D=M_1M_2\overline{M_2}$. By \cite[ Lemma 5.4]{Gall-Rahav13}, we have to consider four situations:\\
(i) $\gcd(\sigma(M^{2h}),D) = 1$, \\
(ii) $\sigma(M^{2h}) = M_1 B$, with $\gcd(B,D)=1$,\\
(iii) $\sigma(M^{2h}) = M_2\overline{M_2} B$, with $\gcd(B,D)=1$,\\
(iv) $\sigma(M^{2h}) = D B$, with $\gcd(B,D)=1$,
where any irreducible divisor of $B$ has degree exceeding $5$.\\
The following lemma contradicts the fact that $U_{2h}$ is a square.
\begin{lemma}
One has $\alpha_3(U_{2h})= 1$ or $\alpha_5(U_{2h})= 1$.
\end{lemma}
\begin{proof}
For (i), (iii) and (iv), use Lemmas 5.9, 5.10, 5.15, 5.17 (in \cite{Gall-Rahav13}).\\
(ii) Since $\sigma(M^{2h}) = (x^2+x+1)B$ and $U_{2h} =  (x^2+x) \sigma(B)$, we obtain (by Lemmas \ref{lesalfal} and \ref{alfasigmM2h}):
$$\left\{\begin{array}{l}
0= \alpha_1(M^{2h}) = \alpha_1(\sigma(M^{2h})) = \alpha_1(B)+1,\\
\alpha_3(U_{2h}) = \alpha_3(\sigma(B)) + \alpha_2(\sigma(B)) = \alpha_3(B)+\alpha_2(B),\\
0 = \alpha_3(M^{2h}) = \alpha_3(\sigma(M^{2h})) = \alpha_3(B)+\alpha_2(B)+\alpha_1(B).
\end{array}
\right.$$
Thus, $\alpha_3(U_{2h}) = \alpha_3(B)+\alpha_2(B) = \alpha_1(B)=1$.
\end{proof}
\subsection{Case II: $M \in \{M_2, \overline{M_2}\}$ and $h \geq 2$} \label{caseM2}
It suffices to consider that $M=M_2=1+x+x^3$.
\begin{lemma} \label{divisordeg} (i) If $h \geq 4$, then $M_1$ $($resp. $\overline{M_2})$ divides $\sigma(M^{2h})$
if and only if $3$ divides $2h+1$ $($resp. $7$ divides $2h+1)$.\\
(ii)  If $h \geq 4$ and if $2h+1$ is divisible by a prime $p \not\in \{3,7\}$, then any irreducible divisor of $\sigma(M^{2h})$ is of degree at least $4$.
\end{lemma}
\begin{proof}
The assertion (ii) follows from (i) which in turn, follows from Corollaries \ref{notdivisor} and \ref{anydivides}.
\end{proof}
We consider three possibilities since $\sigma({M}^{p-1}) = \sigma({M_2}^{2})= M_1 \overline{M_3}$ (product of two Mersenne primes), if $p=3$.

\subsubsection{II-1: $2h+1$ is (divisible by) a prime $p \in \{5,7\}$}
\begin{lemma}
For $p \in \{5,7\}$, some non-Mersenne prime divides $\sigma({M}^{p-1})$.
\end{lemma}
\begin{proof}
Here, $h \in \{2,3\}$. By direct computations,
$U_4 = x^3(x+1)^6 (x^3+x+1)$ and
$U_6 = x^8(x+1)^4(x^3+x+1)^2$ which do not split (despite that $U_6$ is a square).
\end{proof}

\subsubsection{II-2: $2h+1 = 3^w$, for some $w\geq 2$}
In this case, $9$ divides $2h+1$ and $\sigma(M^8)$ divides $\sigma(M^{2h})$ (by Lemma \ref{p-reduction}). But,
$\sigma(M^8) = (x^2+x+1)(x^4+x^3+1)(x^6+x+1)(x^{12}+x^8+x^7+x^4+1)$, where $x^6+x+1 = 1+x(x+1)M_3$ is not a Mersenne prime.
\subsubsection{II-3: $2h+1$ is (divisible by) a prime $p \not\in \{3,5,7\}$}
We may write $p= 2h+1$ with $h \geq 4$.

\begin{lemma}
(i) If $l\in \{1,2,3\}$, then $\alpha_l(U_{2h}) = \alpha_l(\sigma(M^{2h}))$.\\
(ii) If $l\in \{1,2\}$, then $\alpha_l(\sigma(M^{2h})) = \alpha_l(M^{2h})$.\\
(iii) The coefficients $\alpha_3(\sigma(M^{2h}))$ and $\alpha_3(M^{2h}+M^{2h-1})$ are equal.
\end{lemma}
\begin{proof}
(i) It follows from Lemma \ref{divisordeg}.\\
For $l \leq 2$, $6h -l = \deg(\sigma(M^{2h}))-l = \deg((M^{2h})-l > 3(2h-1) = \deg(M^{2h-1})$ and for $3 \leq l\leq 5$, $6h -l > 3(2h-2) = \deg(M^{2h-2})$. Hence, we get (ii) and (iii).
\end{proof}
\begin{corollary}
The coefficient $\alpha_3(U_{2h})$ equals $1$.
\end{corollary}
\begin{proof}
The previous lemma implies that $\alpha_3(U_{2h}) = \alpha_3(M^{2h}+M^{2h-1}) = \alpha_3[(x^3+x)M^{2h-1}] = \alpha_3(M^{2h-1}) + \alpha_1(M^{2h-1})$.\\
But, $M^{2h-1} = (x^3+x+1)^{2h-1} = (x^3+x)^{2h-1} + (x^3+x)^{2h-2} +\cdots$\\
The coefficient of $x^{6h-6}$ (resp. of $x^{6h-4}$) in $M^{2h-1}$ is exactly $\alpha_3(M^{2h-1})$ (resp. $\alpha_1(M^{2h-1})$).
So, $\alpha_3(M^{2h-1})=1$ and $\alpha_1(M^{2h-1})=0$.
\end{proof}

\subsection{Case III: $M \not\in {\mathcal{M}}$}
Here, we have two possibilities.

\subsubsection{III-1: the prime $p$ is such that $ord_p(2) \equiv 0 \mod 8$} \label{case-p-Fermat}
Lemmas \ref{mersennedeg8k} and \ref{ord(2)divise} imply Corollary \ref{pFermat}.
\begin{lemma} \label{mersennedeg8k}
There exists no Mersenne prime of degree multiple of $8$.
\end{lemma}

\begin{proof}
If $Q=1+x^{c_1}(x+1)^{c_2}$ with $c_1+c_2 = 8k$, then $\omega(Q)$ is even by \cite[Corollary 3.3]{Gall-Rahav-mersenn}. So, $Q$ is reducible.
\end{proof}

\begin{corollary} \label{pFermat}
If $ord_p(2) \equiv 0 \mod 8$, then $\sigma(M^{2h})$ is divisible by a non-Mersenne prime.
\end{corollary}
\begin{proof}
Suppose that $\displaystyle{\sigma(M^{2h}) = \prod_{j\in J} P_j}$, where each $P_j$ is a Mersenne prime. Then, Lemma \ref{ord(2)divise} implies that $ord_p(2)$ divides $\deg(P_j)$, for any $j \in J$. So, $8$ divides $\deg(P_j)$. It contradicts Lemma \ref{mersennedeg8k}.
\end{proof}
\subsubsection{III-2: $p$ is a Mersenne prime number with $p \not=7$} \label{mersnumber}
Set $p=2^m-1$, with $m$ and $p$ are both prime.
Note that there are (at present) $51$ known Mersenne prime numbers (OEIS Sequences A$000043$ and A$000668$). The first five of them are: $3,7,31, 127$ and $8191$.
\begin{lemma}
If $p \geq 31$ is a Mersenne prime number, then $\sigma(M^{p-1})$ is divisible by a non-Mersenne prime.
\end{lemma}
\begin{proof}
Here, $a+b = \deg(M) \geq 5$ since $M \not\in {\mathcal{M}}$.
We get our result from Corollary \ref{anydivides} and Lemma \ref{phiandirreduc}-(iii).
\end{proof}
It remains then the case $p=3$ (since $p \not= 7$, in this section).
Lemma \ref{oldresult3} has already treated the case where $\omega(\sigma(M^{2}))=2$. So, we suppose that $\omega(\sigma(M^{2})) \geq 3$. Put
$\sigma(M^{2}) = M_1 \cdots M_r, \ r \geq 3 \text{ and } U_2=\sigma(\sigma(M^2)).$\\
We shall prove that $\alpha_3(U_2) = 1$ (Corollary \ref{corollp=3}), a contradiction to the fact that $U_2$ is a square.
Corollary \ref{reduction4} gives

\begin{lemma} \label{smalldivisors}
(i) The trinomial $1+x+x^2$ divides $\sigma(M^2)$.\\
(ii) No irreducible polynomial of degree $r \geq 3$ such that $2^r-1$ is prime, divides $\sigma(M^2)$.
\end{lemma}

\begin{corollary} \label{sigmaM2}
The polynomial $\sigma(M^2)$ is of the form $(1+x+x^2) B$, where $\gcd(1+x+x^2,B) = 1$ and any
prime divisor of $B$ has degree at least $4$.
\end{corollary}

\begin{lemma} \label{lesalphasigmaM2}
If $\sigma(M^2) = (1+x+x^2) B$ with $\gcd(1+x+x^2,B) = 1$, then
$$\left\{\begin{array}{l}
(i) \ \alpha_1(\sigma(M^2)) = \alpha_1(B)+1,\ \alpha_2(\sigma(M^2)) = \alpha_2(B)+\alpha_1(B)+1,\\
(ii) \ \alpha_3(\sigma(M^2)) = \alpha_3(B)+\alpha_2(B)+\alpha_1(B),\\
(iii) \ \alpha_3(\sigma(M^2)) = 0.
\end{array}
\right.$$
\end{lemma}
\begin{proof}
We directly get (i) and (ii). For (iii), $\sigma(M^2)=1+M+M^2=x^{2a}(x+1)^{2b} + x^a(x+1)^b+1$. Moreover,
$2a+2b-3 > a+b$ because $a+b \geq 4$ and $x^{2a}(x+1)^{2b}$ is a square. So, $\alpha_3(\sigma(M^2))=\alpha_3(x^{2a}(x+1)^{2b})=0$.
\end{proof}
\begin{lemma} \label{lesalphaW}
Some coefficients of $U_2$ and $B$ satisfy:
$$\alpha_1(U_2) = \alpha_1(B)+1,\ \alpha_2(U_2)=\alpha_2(B) + \alpha_1(B),\ \alpha_3(U_2)=\alpha_3(B) + \alpha_2(B).$$
\end{lemma}
\begin{proof}
Corollary \ref{sigmaM2} implies that $U_2=\sigma(\sigma(M^2))=\sigma((1+x+x^2) B)=\sigma(1+x+x^2)\sigma(B)=(x^2+x) \sigma(B)$.
Any irreducible divisor of $B$ has degree more than $3$. Hence, $\alpha_l(\sigma(B)) = \alpha_l(B),$ for $1\leq l \leq 3$. \\
One gets:
$\left\{\begin{array}{l}
\alpha_1(U_2)=\alpha_1(\sigma(B))+1=\alpha_1(B)+1,\\
\alpha_2(U_2)=\alpha_2(\sigma(B))+\alpha_1(\sigma(B)) = \alpha_2(B)+\alpha_1(B),\\
\text{$\alpha_3(U_2)=\alpha_3(\sigma(B))+\alpha_2(\sigma(B)) = \alpha_3(B) + \alpha_2(B)$.}
\end{array}
\right.$
\end{proof}

\begin{corollary} \label{corollp=3}
The coefficient $\alpha_3(U_2)$ equals $1$.
\end{corollary}
\begin{proof}
The polynomial $U_2$ is a square, so $0=\alpha_1(U_2)=\alpha_1(B)+1$ and thus $\alpha_1(B)=1$.
Lemma \ref{lesalphasigmaM2}-(iii) implies that $0=\alpha_3(\sigma(M^2)) = \alpha_3(B)+\alpha_2(B)+\alpha_1(B)$. Therefore, $\alpha_3(U_2)= \alpha_3(B)+\alpha_2(B) = \alpha_1(B)=1.$
\end{proof}

\begin{remark}
\emph{Our method fails for  $p=7$. Indeed, for many $M$, one has $\alpha_3(U_6) = \alpha_5(U_6) =0$. So, we do not reach a contradiction. We should find a large enough odd integer $l$ such that, $\alpha_l(U_6) =0$. But, this does not appear always possible.}
\end{remark}

\def\biblio{\def\titrebibliographie{References}\thebibliography}
\let\endbiblio=\endthebibliography




\newbox\auteurbox
\newbox\titrebox
\newbox\titrelbox
\newbox\editeurbox
\newbox\anneebox
\newbox\anneelbox
\newbox\journalbox
\newbox\volumebox
\newbox\pagesbox
\newbox\diversbox
\newbox\collectionbox
\def\fabriquebox#1#2{\par\egroup
\setbox#1=\vbox\bgroup \leftskip=0pt \hsize=\maxdimen \noindent#2}
\def\bibref#1{\bibitem{#1}


\setbox0=\vbox\bgroup}
\def\auteur{\fabriquebox\auteurbox\styleauteur}
\def\titre{\fabriquebox\titrebox\styletitre}
\def\titrelivre{\fabriquebox\titrelbox\styletitrelivre}
\def\editeur{\fabriquebox\editeurbox\styleediteur}

\def\journal{\fabriquebox\journalbox\stylejournal}

\def\volume{\fabriquebox\volumebox\stylevolume}
\def\collection{\fabriquebox\collectionbox\stylecollection}
{\catcode`\- =\active\gdef\annee{\fabriquebox\anneebox\catcode`\-
=\active\def -{\hbox{\rm
\string-\string-}}\styleannee\ignorespaces}}
{\catcode`\-
=\active\gdef\anneelivre{\fabriquebox\anneelbox\catcode`\-=
\active\def-{\hbox{\rm \string-\string-}}\styleanneelivre}}
{\catcode`\-=\active\gdef\pages{\fabriquebox\pagesbox\catcode`\-
=\active\def -{\hbox{\rm\string-\string-}}\stylepages}}
{\catcode`\-
=\active\gdef\divers{\fabriquebox\diversbox\catcode`\-=\active
\def-{\hbox{\rm\string-\string-}}\rm}}
\def\ajoutref#1{\setbox0=\vbox{\unvbox#1\global\setbox1=
\lastbox}\unhbox1 \unskip\unskip\unpenalty}
\newif\ifpreviousitem
\global\previousitemfalse
\def\separateur{\ifpreviousitem {,\ }\fi}
\def\voidallboxes
{\setbox0=\box\auteurbox \setbox0=\box\titrebox
\setbox0=\box\titrelbox \setbox0=\box\editeurbox
\setbox0=\box\anneebox \setbox0=\box\anneelbox
\setbox0=\box\journalbox \setbox0=\box\volumebox
\setbox0=\box\pagesbox \setbox0=\box\diversbox
\setbox0=\box\collectionbox \setbox0=\null}
\def\fabriquelivre
{\ifdim\ht\auteurbox>0pt
\ajoutref\auteurbox\global\previousitemtrue\fi
\ifdim\ht\titrelbox>0pt
\separateur\ajoutref\titrelbox\global\previousitemtrue\fi
\ifdim\ht\collectionbox>0pt
\separateur\ajoutref\collectionbox\global\previousitemtrue\fi
\ifdim\ht\editeurbox>0pt
\separateur\ajoutref\editeurbox\global\previousitemtrue\fi
\ifdim\ht\anneelbox>0pt \separateur \ajoutref\anneelbox
\fi\global\previousitemfalse}
\def\fabriquearticle
{\ifdim\ht\auteurbox>0pt        \ajoutref\auteurbox
\global\previousitemtrue\fi \ifdim\ht\titrebox>0pt
\separateur\ajoutref\titrebox\global\previousitemtrue\fi
\ifdim\ht\titrelbox>0pt \separateur{\rm in}\
\ajoutref\titrelbox\global \previousitemtrue\fi
\ifdim\ht\journalbox>0pt \separateur
\ajoutref\journalbox\global\previousitemtrue\fi
\ifdim\ht\volumebox>0pt \ \ajoutref\volumebox\fi
\ifdim\ht\anneebox>0pt  \ {\rm(}\ajoutref\anneebox \rm)\fi
\ifdim\ht\pagesbox>0pt
\separateur\ajoutref\pagesbox\fi\global\previousitemfalse}
\def\fabriquedivers
{\ifdim\ht\auteurbox>0pt
\ajoutref\auteurbox\global\previousitemtrue\fi
\ifdim\ht\diversbox>0pt \separateur\ajoutref\diversbox\fi}
\def\endbibref
{\egroup \ifdim\ht\journalbox>0pt \fabriquearticle
\else\ifdim\ht\editeurbox>0pt \fabriquelivre
\else\ifdim\ht\diversbox>0pt \fabriquedivers \fi\fi\fi.\voidallboxes}

\let\styleauteur=\sc
\let\styletitre=\it
\let\styletitrelivre=\sl
\let\stylejournal=\rm
\let\stylevolume=\bf
\let\styleannee=\rm
\let\stylepages=\rm
\let\stylecollection=\rm
\let\styleediteur=\rm
\let\styleanneelivre=\rm

\begin{biblio}{99}

\begin{bibref}{Beard2}
\auteur{J. T. B. Beard Jr}  \titre{Perfect polynomials revisited}
\journal{Publ. Math. Debrecen} \volume{38/(1-2)} \pages 5-12 \annee
1991
\end{bibref}

\begin{bibref}{BeardU}
\auteur{J. T. B. Beard Jr} \titre{Unitary perfect polynomials over
$GF(q)$} \journal{Rend. Accad. Lincei} \volume{62} \pages 417-422
\annee 1977
\end{bibref}

\begin{bibref}{BeardU2}
\auteur{J. T. B. Beard Jr, A. T. Bullock, M. S. Harbin}
\titre{Infinitely many perfect and unitary perfect polynomials}
\journal{Rend. Accad. Lincei} \volume{63} \pages 294-303 \annee 1977
\end{bibref}

\begin{bibref}{Beard}
\auteur{J. T. B. Beard Jr, J. R. Oconnell Jr, K. I. West}
\titre{Perfect polynomials over $GF(q)$} \journal{Rend. Accad.
Lincei} \volume{62} \pages 283-291 \annee 1977
\end{bibref}

\begin{bibref}{Canaday}
\auteur{E. F. Canaday} \titre{The sum of the divisors of a
polynomial} \journal{Duke Math. J.} \volume{8} \pages 721-737 \annee
1941
\end{bibref}

\begin{bibref}{Gall-Rahav5}
\auteur{L. H. Gallardo, O. Rahavandrainy} \titre{Even perfect
polynomials over $\F_2$ with four prime factors} \journal{Intern. J.
of Pure and Applied Math.} \volume{52(2)} \pages 301-314 \annee 2009
\end{bibref}

\begin{bibref}{Gall-Rahav8}
\auteur{L. H. Gallardo, O. Rahavandrainy} \titre{All perfect
polynomials with up to four prime factors over $\F_4$}
\journal{Math. Commun.} \volume{14(1)} \pages 47-65 \annee 2009
\end{bibref}

\begin{bibref}{Gall-Rahav12}
\auteur{L. H. Gallardo, O. Rahavandrainy} \titre{On even (unitary) perfect
polynomials over $\F_{2}$ } \journal{Finite Fields Appl.} \volume{18} \pages 920-932 \annee 2012
\end{bibref}

\begin{bibref}{Gall-Rahav13}
\auteur{L. H. Gallardo, O. Rahavandrainy} \titre{Characterization of Sporadic perfect
polynomials over $\F_{2}$ } \journal{Functiones et Approx.} \volume{55(1)} \pages 7-21 \annee 2016
\end{bibref}

\begin{bibref}{Gall-Rahav-mersenn}
\auteur{L. H. Gallardo, O. Rahavandrainy} \titre{On Mersenne
polynomials over $\F_{2}$} \journal{Finite Fields Appl.}
\volume{59} \pages 284-296 \annee 2019
\end{bibref}

\begin{bibref}{Rudolf}
\auteur{R. Lidl, H. Niederreiter} \titrelivre{Finite Fields,
Encyclopedia of Mathematics and its applications} \editeur{Cambridge
University Press} \anneelivre 1983 (Reprinted 1987)
\end{bibref}

\begin{bibref}{Rahav}
\auteur{O. Rahavandrainy} \titre{Familles de polyn\^omes unitairement parfaits sur $\F_2$} \journal{C. R. Math. Acad. Sci. Paris}
\volume{359(2)} \pages 123-130 \annee 2021
\end{bibref}

\end{biblio}

\end{document}